\documentclass[11pt, reqno]{amsart}
\usepackage{amsfonts}
\usepackage{amsmath}
\usepackage{amssymb}
\usepackage{amsthm}
\usepackage{enumerate}
\usepackage{mathdots}
\usepackage{hyperref}
\usepackage{caption}
\usepackage{adjustbox}
\usepackage{bbold}
\usepackage{mathtools}
\usepackage{hyperref}
\newtheorem{theorem}{Theorem}[section]
\newtheorem{lemma}[theorem]{Lemma}
\newtheorem{proposition}[theorem]{Proposition}

\newtheorem{cor}[theorem]{Corollary}
\newtheorem{prob}[theorem]{Problem}
\newtheorem*{rmk*}{Remark}

\usepackage{fancyhdr}

\title{An improved lower bound for a problem of Littlewood on the zeros of cosine polynomials}
\author{Benjamin Bedert}
\thanks{benjamin.bedert@maths.ox.ac.uk\\The author gratefully acknowledges financial support from the EPSRC}
\begin{document}
\begin{abstract}
Let $Z(N)$ denote the minimum number of zeros in $[0,2\pi]$ that a cosine polynomial of the form $$f_A(t)=\sum_{n\in A}\cos nt$$ can have when $A$ is a finite set of non-negative integers of size $|A|=N$. It is an old problem of Littlewood to determine $Z(N)$. In this paper, we obtain the lower bound $Z(N)\geqslant (\log\log N)^{(1+o(1))}$ which exponentially improves on the previous best bounds of the form $Z(N)\geqslant (\log\log\log N)^c$ due to Erd\'elyi and Sahasrabudhe.
\end{abstract}
\maketitle

\tableofcontents
\section{Introduction}
Littlewood posed the following problem in his 1968 paper “Some Problems in Real and Complex
Analysis” \cite[Problem 22]{littlewoodproblems}.
\begin{prob} \label{littprob} If $A$ is a set of non-negative integers of size $|A|=N$, what is the lower bound on the number of real zeros of $\sum_{n\in A} \cos nt$ in a period $[0,2\pi]$? Possibly $N-1$, or not much
less.
\label{Littlewoodprob}
\end{prob}
The first progress on this problem was made by Borwein, Erd\'elyi, Ferguson and Lockhart \cite{borwein} who proved the existence of cosine polynomials $\sum_{n\in A} \cos nt$ with no more than $O(N^{5/6}\log N)$ roots in a period, hence giving a counterexample to Littlewood's suggested lower bound. Their construction was later optimised by Konyagin \cite{konyaginzeros} and independently by Ju\v skevi\v cius and Sahasrabudhe \cite{juskevicius} to obtain the better upper bound $O((N\log N)^{2/3})$. Let us write $Z(f_A)$ for the number of zeros in $[0,2\pi]$ of the cosine polynomial

\begin{equation}
    f_A(t)=\sum_{n\in A} \cos nt
\label{1cosinepoly}
\end{equation}
where $A$ is a set of non-negative integers. Obtaining lower bounds for $Z(f_A)$ seems to be a hard problem and it was only recently established that the number of such zeros grows to infinity as $|A|\to \infty$. This result was proved independently by Erd\'elyi \cite{erdelyioriginal}, and by Sahasrabudhe \cite{sahasrabudhe} who further obtained the explicit lower bound $Z(f_A)\geqslant(\log\log\log |A|)^{1/2-o(1)}$. By combining the arguments of \cite{erdelyioriginal} and \cite{sahasrabudhe}, Erd\'elyi \cite{erdelyi} later obtained the slight improvement $Z(f_A)\geqslant (\log\log\log |A|)^{1-o(1)}$. We define
\begin{equation}
    Z(N)=\min_{A\subset \mathbf{N}: |A|=N} Z(f_A)
\label{littlewoodques}
\end{equation} so that Littlewood's Problem \ref{Littlewoodprob} is precisely to determine $Z(N)$. Then the best bounds prior to this work state that \begin{equation}
    (\log \log \log N)^{1-o(1)}\leqslant  Z(N)\leqslant (N\log N)^{2/3},
\label{Zbounds}
\end{equation}
and there remains a large gap between the upper and lower bounds.  The methods of \cite{erdelyioriginal,erdelyi,sahasrabudhe} further prove lower bounds on the number of zeros of a more general class of cosine polynomials. Let $S$ be a finite set and let
\begin{equation}
    g(t) = \sum_{n=0}^N a_n \cos nt
\label{cosinepolygene}
\end{equation}
be a cosine polynomial with coefficients $a_n\in S$. The following lower bound is proved in \cite{erdelyi} (following similar theorems in \cite{erdelyioriginal,sahasrabudhe}).
\begin{theorem}[\cite{erdelyi}, Theorem 2.1] \label{erdetheo}
    Let $S\subset \mathbf{Z}$ be finite and $M(S)=\max_{s\in S} |s|$.
    Let $g$ be a cosine polynomial as in \eqref{cosinepolygene} with coefficients $a_n\in S$. Then the number of roots of $g$  satisfies
    \begin{equation*}
        Z(g)\geqslant \left(\frac{c}{1+\log M(S)}\right)\frac{\log\log\log |g(0)|}{\log\log\log\log |g(0)|}-1,
    \end{equation*}
    where $c>0$ is an absolute constant.
\label{sahaerde}
\end{theorem}
In particular, this theorem shows that if $S\subset\mathbf{N}$ is finite, then the number of zeros of a cosine polynomial $g$ with coefficients in $S$ tends to infinity with the degree $\deg(g)$. Interestingly, it is known that such a result does not hold for general $S\subset\mathbf{Z}$; in \cite{sahasrabudhe} it was shown to be false for $S=\{-1,0,1,2\}$. It was conjectured however that this is true for $S=\{-1,1\}$ which is the simplest case to which Theorem \ref{sahaerde} does not apply because $|g(0)|$ may remain bounded as $\deg(g)\to\infty$. This conjecture was proved recently in \cite{bedert}, thus providing a qualitative extension of Theorem \ref{sahaerde}. The purpose of this paper is to obtain the following quantitative improvement over Theorem \ref{sahaerde}.
\begin{theorem}
    There exists an absolute constant $c>0$ such that the following holds. Let $S\subset \mathbf{Z}$ be finite and $M(S)=\max_{s\in S} |s|$.
    Let $g$ be a cosine polynomial as in \eqref{cosinepolygene} with coefficients $a_n\in S$. Then the number of roots of $g$  satisfies
    \begin{equation*}
        Z(g)\geqslant \left(\frac{c}{1+\log M(S)}\right)\frac{\log\log |g(0)|}{\log\log\log |g(0)|}.
    \end{equation*}
\label{newtheorem}
\end{theorem}
Applying this theorem  with $S=\{0,1\}$ clearly yields the following improved lower bound compared to \eqref{Zbounds} for Littlewood's Problem \ref{Littlewoodprob}.
\begin{theorem}
    We have that $Z(N)\geqslant (\log\log N)^{1-o(1)}$.
\end{theorem}

\textbf{Acknowledgements}
The author would like to thank Thomas Bloom and Ben Green for useful discussions and comments on earlier versions of the paper. The author also gratefully acknowledges financial support from the EPSRC.
\section{Notation, prerequisites and organisation of the paper}
$\mathbf{N},\mathbf{Z}$ and $\mathbf{R}$ denote the sets of non-negative integers, the integers and the real numbers respectively. We use the asymptotic notation $f=O(g)$ or $f\ll g$ if there is a constant $C$ such that $|f(x)|\leqslant C g(x)$ for all $x$. For a real number $t$, we use the notation $e(t)=e^{2\pi i t}$. For a $1$-periodic function $g$ we denote its $L^1$-norm by $\lVert g\rVert_1\vcentcolon= \int_0^1|g(t)|\,dt$.

\bigskip

The rough approach in Erd\'elyi and Sahasrabudhe's arguments is to first prove that cosine polynomials with few roots must be very structured, and then prove that structured polynomials still have many roots. In \hyperref[sec:3]{Section 3} we prove a result which provides an exponential improvement for the number of roots of structured cosine polynomials. This is almost strong enough to obtain Theorem \ref{newtheorem} by using Erd\'elyi's structural result \cite[Lemma 3.9]{erdelyi}. To actually obtain Theorem \ref{newtheorem}, we need to prove a slightly better structure theorem than is given in Erd\'elyi or Sahasrabudhe's papers. This is achieved in \hyperref[sec:4]{Section 4} with an argument that is based on the ideas in \cite{erdelyi,sahasrabudhe}. In \hyperref[sec:5]{Section 5} we combine both results to yield the improved Theorem \ref{newtheorem}.

\bigskip

A key role in the argument will be played by the following result known as ‘the Littlewood $L^1$ conjecture', which was established by Konyagin \cite{konyaginl} and independently by McGehee, Pigno and Smith \cite{mcgeheepignosmith} who proved the following generalisation.
\begin{theorem}[Littlewood's $L^1$ conjecture]
Let $a_1,a_2,\dots,a_k$ be complex numbers and $n_1<n_2<\dots<n_k$ integers. Then $$\left|\left|\sum_{j=1}^k a_je(n_jt)\right|\right|_1\gg \sum_{j=1}^k\frac{|a_j|}{j}.$$
\label{Litllconj}    
\end{theorem}

\section{$L^1$ bounds for structured trigonometric polynomials}
\label{sec:3}

In this section we prove the following theorem which provides an improved bound for the number of roots of structured trigonometric polynomials.
\begin{proposition}
    Let $S\subset \mathbf{Z}$ be a finite set and $M(S)=\max_{s\in S} |s|$.
    Let $g(t) = \sum_{n=0}^N a_n \cos nt$ be a cosine polynomial with coefficients $a_n\in S$. Suppose that $\deg g= N$ and that we can partition $[0,N]=\cup_{j=1}^KI_j$ into $K$ intervals such that on each interval $I_j$, the coefficient sequence $(a_n)_{n\in I_j}$ is periodic with period $P$. If $d$ is the number of sign changes of $g$ in $[0,2\pi]$, then
    \begin{equation}
        dM(S)P^2\log (KP)\gg \log |g(0)|,
    \label{structurerootsprop}
    \end{equation}
    where the implied constant is absolute.
\label{improvetheo}
\end{proposition}
To begin, we may rescale the variable $t$ by $2\pi$ and write $g$ in the following form
\begin{equation}
    g(t)=\sum_{n=0}^Na_n\cos(2\pi nt)=\sum_{n=-N}^N\hat{g}(n)e(nt)
\end{equation} where $e(x)\vcentcolon= e^{2\pi i x}$ and $\hat{g}$ denotes the Fourier transform of $g$. We define the related cosine polynomial
\begin{equation}
    \tilde{g}(t)=g(t)\left(\frac{2}{P}\sum_{n=0}^{P-1} \cos 2\pi nt\right)^2
\label{gtilde}
\end{equation}
and record some important properties. Note that $\tilde{g}$ and $g$ have exactly the same $d$ sign changes in $[0,1]$, and from the formula $2\cos\alpha \cos \beta = \cos(\alpha-\beta)+\cos(\alpha+\beta)$ it follows that $\tilde{g}$ is a cosine polynomial with coefficients in $(P^{-2}\cdot\mathbf{Z})\cap [-4M(S),4M(S)]$, recalling that $|a_n|\leqslant M(S)$ by definition. It is also clear that $|\tilde{g}(0)|= 4|g(0)|$. Finally, we analyse the structure of $\tilde{g}$. By assumption, we can write
\begin{equation*}
    g(t)=\sum_{j=1}^K\sum_{n\in I_j} a_n\cos(2\pi nt)
\end{equation*}
where each of the sequences $(a_n)_{n\in I_j}$ is periodic with period $P$. For each $j\in[K]$, we split the interval $I_i=[u_i,v_i]$ into the interval $J_i=[u_i+2P,v_i-2P]$ and the set of remaining elements $I_i\setminus J_i$ (there may exist $i$ for which $J_i$ is empty). Define $S_j$ to be the sum of any $P$ consecutive terms in the sequence $(a_n)_{n\in I_j}$, so note that $S_j$ is a constant only depending on $j$ as $(a_n)_{n\in I_j}$ is $P$-periodic. Hence, if $j\in[K]$ and $n\in J_j$, then \eqref{gtilde} and the expansion 
\begin{align*}
    \left(\frac{2}{P}\sum_{n=0}^{P-1} \cos 2\pi nt\right)^2=\frac{1}{P^2}\left(1+2\sum_{m=-P+1}^{P-1}e(mt)+\sum_{m=-2P+1}^{2P-1}(2P-1-|m|)e(mt)\right),
\end{align*} give
\begin{align*}
    \hat{\tilde{g}}(n)&=\frac{1}{2P^{2}}\left(a_n+2\sum_{m=-P+1}^{P-1}a_{n+m}+\sum_{m=-2P+2}^{2P-2}(2P-1-|m|)a_{n+m}\right)\\
    &=\frac{1}{2P^{2}}\left(4P\sum_{m=0}^{P-1}a_{n+m}\right)=\frac{2S_j}{P}
\end{align*}
since $J_j+[-2P,2P]\subset I_j$ so that all the terms $a_{n+m}$ appearing above are part of the periodic sequence $(a_n)_{n\in I_j}$. Therefore we may write
\begin{align*}
    \tilde{g}(t)=\sum_{j=1}^K\frac{2S_j}{P}\sum_{n\in J_j} \cos(2\pi nt)+\sum_{n\in [N+2P]\setminus \cup_j J_j}2\hat{\tilde{g}}(n)\cos(2\pi nt).
\end{align*}
Note that the sum on the right hand side has at most $O(PK)$ terms since there are $K$ intervals $I_j$ and each of these contributes at most $|I_j\setminus J_j|\leqslant 4P$ terms to this second sum. Hence, if we consider all these terms as a sum over $O(PK)$ intervals of length $1$, then we can find 
\begin{equation}
    \tilde{K}=K+O(PK)=O(PK)
    \label{tildeKbound}
\end{equation} intervals $\tilde{J}_1,\tilde{J}_2,\dots,\tilde{J}_{\tilde{K}}$ partitioning $[N+2P]$ and constants $c_j$ which satisfy $c_j\in (P^{-2}\cdot\mathbf{Z})\cap [-4M(S),4M(S)]$ so that
\begin{align}
    \tilde{g}(t)=\sum_{j=1}^{\tilde{K}}c_j\sum_{n\in \tilde{J}_j} \cos(2\pi nt).
\label{tildegpart}
\end{align}
The proof of Proposition \ref{improvetheo} comes down to finding two bounds for the $L^1$-norm of $\tilde{g}$.
The lower bound will follow from an application the Littlewood $L^1$ conjecture \ref{Litllconj}.
\begin{lemma}\label{lem:L1lowerbound}
    Let $\tilde{g}$ be as in \eqref{tildegpart}. Then \begin{equation}
        \lVert\tilde{g}\rVert_1\gg P^{-2}\log\frac{|g(0)|}{M(S)}.
        \label{L1lower}
    \end{equation}
\end{lemma}
\begin{proof}
    Recall that the coefficients $c_j$ lie in $P^{-2}\cdot\mathbf{Z}$ so that each non-zero $c_j$ satisfies $|c_j|\geqslant P^{-2}$. The bound in Theorem \ref{Litllconj} therefore immediately gives
\begin{align*}
    \lVert\tilde{g}\rVert_1&\gg P^{-2}\log\left(\sum_{j:c_j\neq 0} |\tilde{J}_j|\right).
\end{align*}
We also have that $|c_j|\leqslant 4M(S)$ which yields the following lower bound for the number of non-zero Fourier coefficients of $\tilde{g}$:
\begin{align*}
    4M(S)\sum_{j: c_j\neq 0}|\tilde{J}_j|&\geqslant \left|\sum_{j: c_j\neq 0}c_j|\tilde{J}_j|\right|\\
    &=|\tilde{g}(0)|.
\end{align*}
Hence, $\sum_{j:c_j\neq 0} |\tilde{J}_j|\geqslant \frac{|\tilde{g}(0)|}{4M(S)}=\frac{|g(0)|}{M(S)}$ since $\tilde{g}(0)=4g(0)$ and plugging this in the lower bound for $\lVert \tilde{g}\rVert_1$ above finishes the proof.
\end{proof} 
We now move on to obtaining an upper bound for $\lVert\tilde{g}\rVert_1$. We assumed that $g$ has $d$ sign changes in $(0,1/2)$ and hence so does $\tilde{g}$. Let us denote these by $\xi_1,\xi_2,\dots,\xi_d$. If we write $\xi_{d+1}=1/2$, then as $\tilde{g}$ is even we get
\begin{align*}
    \lVert\tilde{g}\rVert_1&=2\int_0^{1/2}|\tilde{g}(t)|\,dt\\
    &= 2\sum_{m=1}^d \varepsilon_m\int_{\xi_m}^{\xi_{m+1}}\tilde{g}(t)\,dt,
\end{align*}
for some signs $\varepsilon_m\in\{-1,1\}$. We define the function $G(x)\vcentcolon=\int_0^x\tilde{g}(t)\,dt$ and hence obtain the following upper bound
\begin{align}
    \lVert \tilde{g}\rVert_1\leqslant 4d\sup_{x\in[0,1/2]}|G(x)|.
    \label{L1upper}
\end{align}
This bound above can be quite wasteful for general $P$ and $K$ and could be improved, but the loss of a factor of $d$ is unimportant for \eqref{structurerootsprop} in the situation provided by the structural result Proposition \ref{smallperiodprop} where $P,\log K$ are exponential in $d$.
\begin{lemma}
Let $D_n(t)=1/2+\sum_{m=1}^n \cos(mt)$ be a Dirichlet kernel. Then uniformly for $x\in[0,\pi]$, we have
\begin{equation}
    \int_0^x D_n(t)\,dt = \int_0^x \frac{\sin nt}{t} \,dt +O\left(\frac{1}{n}\right).
\end{equation}
\label{D_napprox}
\end{lemma}
The proof can be assembled by combining some results from chapter II of Zygmund's book \cite{zygmund}.
\begin{proof}
    Let $D^*_n(t)= D_n(t)-1/2\cos(nt)$. Since $\int_J \cos(nt)\,dt=O(1/n)$ uniformly for all intervals $J\subset[0,2\pi]$, it suffices to prove that $\int_0^xD^*_n(t)\,dt =\int_0^x \frac{\sin nt}{t} \,dt +O(1/n)$. We have the standard formula  $D^*_n(t)=\frac{1}{2}\cot(t/2)\sin(nt)$, see \cite[(5.2)]{zygmund}, and hence we obtain
    \begin{align*}
        \int_0^xD^*_n(t)-\frac{\sin(nt)}{t}\,dt = \int_0^{2\pi}H_x(t)\sin(nt)\,dt
    \end{align*}
    where $H_x(t)\vcentcolon= \begin{cases}
\frac{1}{2}\cot(t/2)-1/t &\text{if $t\in (0,x)$}\\
0 &\text{if $t \in[x,2\pi]$}
\end{cases}$. Let $f$ be a function on $(0,2\pi)$ whose total variation $V(f)$ is bounded, then we can bound the Fourier coefficients of $f$ in terms of $V(f)$ as follows, see \cite[Theorem 4.12]{zygmund}:
\begin{equation}
    \hat{f}(n)\ll \frac{V(f)}{n}.
\label{variationbound}
\end{equation}
Using the fact that $H_\pi(t)$ is bounded, vanishes at $t=0$ and has bounded variation on $[0,\pi]$, we see that $V(H_x)=O(1)$ is uniformly bounded for all $x\in[0,\pi]$. Hence, from \eqref{variationbound} we deduce that uniformly for $x\in[0,\pi]$ we have
\begin{align*}
    \int_0^{2\pi}H_x(t)\sin(nt)\,dt \ll |\hat{H_x}(n)|\ll n^{-1}.
\end{align*}
\end{proof}
Let $\operatorname{Si}(x)\vcentcolon=\int_0^x\frac{\sin t}{t}\,dt$ denote the sine integral. Then by Lemma \ref{D_napprox} we get
\begin{align*}
    \int_0^xD_n(2\pi t)\,dt = \frac{1}{2\pi}\operatorname{Si}(2\pi nx)+O(1/n)
\end{align*}
uniformly for all $x\in[0,1/2]$. Getting back to $\tilde{g}$, if we write the intervals $\tilde{J}_j=(n_{j-1},n_j]$, then we can rewrite \eqref{tildegpart} as
\begin{equation}
    \tilde{g}(t)=\sum_{j=1}^{\tilde{K}}c_j(D_{n_j}(2\pi t)-D_{n_{j-1}}(2\pi t))
\end{equation}
so that for $x\in[0,1/2]$ we have
\begin{align}
    G(x)&=\int_0^x\tilde{g}(t)\,dt\nonumber\\
    &=\frac{1}{2\pi}\sum_{j=1}^{\tilde{K}}c_j\left(\operatorname{Si}(2\pi n_jx)-\operatorname{Si}(2\pi n_{j-1}x)\right) +O\left(\max_j|c_j|\sum_{j=1}^{\tilde{K}}\frac{1}{n_j}\right)\nonumber\\
    &=\frac{1}{2\pi}\sum_{j=1}^{\tilde{K}}c_j\int_{2\pi n_{j-1}x}^{2\pi n_jx}\frac{\sin t}{t}\,dt+O\left(M(S)\log(\tilde{K})\right),
\label{Gbound1}
\end{align}
where we used that $\max|c_j|\leqslant 4M(S)$ and that $n_1<n_2<\dots<n_{\tilde{K}}$ to bound the error term. We have therefore shown that
\begin{align}
    \sup_{x\in[0,1/2]}|G(x)|\ll M(S)\sup_{x\in[0,\pi]}\sum_{j=1}^{\tilde{K}}\left|\int_{n_{j-1}x}^{n_jx}\frac{\sin t}{t}\,dt\right|+O(M(S)\log\tilde{K}).
\label{Gbound2}
\end{align}
The first term in \eqref{Gbound2} will be bounded with the help of the next lemma.
\begin{lemma}
    There exists an absolute constant $C>0$ such that for all $y'>y>1$ we have
    \begin{equation*}
        \left|\int_y^{y'}\frac{\sin t}{t}\,dt\right|\leqslant Cy^{-1}\min\left(1,|y'-y|\right).
    \end{equation*}
\label{sincbound}
\end{lemma}
\begin{proof}
    From the pointwise bound $|\frac{\sin t}{t}|\leqslant t^{-1}$ we immediately deduce $\left|\int_y^{y'}\frac{\sin t}{t}\,dt\right|\leqslant y^{-1}|y'-y|.$ For the other bound, assume that $a\pi,(a+1)\pi,\dots,b\pi$ are the integer multiples of $\pi$ in $(y,y')$. The sequence $\left(\int_{j\pi}^{(j+1)\pi}t^{-1}\sin t\,dt\right)_j$ is alternating in sign and strictly decreasing in absolute value, so
    \begin{align*}
       \left|\int_y^{y'}\frac{\sin t}{t}\,dt\right|&\leqslant \left|\int_y^{a\pi} t^{-1}\sin t\,dt\right|+\left|\sum_{j=a}^{b-1}\int_{j\pi}^{(j+1)\pi}t^{-1}\sin t\,dt\right|+\left|\int_{b\pi}^{y'}t^{-1}\sin t\,dt\right|\\
       &\ll y^{-1}+\left|\int_{a\pi}^{(a+1)\pi}t^{-1}\sin t\,dt\right|+(y')^{-1}\ll y^{-1}.
    \end{align*}
\end{proof}
Combining \eqref{L1upper} and \eqref{Gbound2}, we have  shown so far that \begin{equation}\label{eq:L1upperbound}
\lVert\tilde{g}\rVert_1\ll  dM(S)\sup_{x\in[0,\pi]}\sum_{j=1}^{\tilde{K}}\left|\int_{n_{j-1}x}^{n_jx}\frac{\sin t}{t}\,dt\right|+dM(S)\log\tilde{K}.
\end{equation}
The final step in proving an upper bound for $\lVert \tilde{g}\rVert_1$ is bounding the first term above. 
\begin{lemma}
\label{lem:Sibound}
    Let $x\in[0,\pi]$ and let $n_1<n_2<\dots<n_{\tilde{K}}$ be positive integers. Then 
    $$\sum_{j=1}^{\tilde{K}}\left|\int_{n_{j-1}x}^{n_jx}\frac{\sin t}{t}\,dt\right|\ll 1+\log \tilde{K}.$$
\end{lemma}
\begin{proof}
    Let $x\in[0,\pi]$ and $n_1<n_2<\dots<n_{\tilde{K}}$. For each $m\in\mathbf{N}$ which satisfies $m<n_{\tilde{K}}x$ we define 
\begin{align*}
    j_m\vcentcolon=\min\{j\in[\tilde{K}]: n_jx>m\}.
\end{align*}
This produces a set of indices $\{j_1<j_2<\dots<j_{\tilde{k}}\}\subset[\tilde{K}]$ where $\tilde{k}$ is the largest integer less than $n_{\tilde{K}}x$. In other words, we define a partition of $[1,\tilde{K}]$ into intervals $[1,j_1),[j_1,j_2),\dots,[j_{\tilde{k}},\tilde{K}]$ such that
\begin{align*}
    n_jx\in \begin{cases}
        [0,1] & \text{ if $j\in [1,j_1)$,}\\
        (m,m+1] & \text{ if $j\in [j_m,j_{m+1})$ for each $m=1,2,\dots,\tilde{k}$,}\\
        (\tilde{k},\tilde{k}+1] & \text{ if $j\in [j_{\tilde{k}},\tilde{K}]$.}
    \end{cases}
\end{align*} Then we can bound the sum of interest as follows
\begin{align}
    \sum_{j=1}^{\tilde{K}}\left|\int_{n_{j-1}x}^{n_jx}\frac{\sin t}{t}\,dt\right|&\leqslant \int_0^1\left|\frac{\sin t}{t}\right|\,dt+\left|\int_{n_{j_1-1}x}^{n_{j_1}x} \frac{\sin t}{t}\,dt\right|\nonumber\\
    &+\sum_{j=j_1}^{j_{\tilde{k}}}\left|\int_{n_{j}x}^{n_{j+1}x}\frac{\sin t}{t}\,dt\right|+\left|\int_{n_{j_{\tilde{k}}}x}^{n_{{\tilde{K}}}x} \frac{\sin t}{t}\,dt\right|,\label{eq:finalstep}
\end{align}
noting that $n_{j_1-1}x\leqslant 1$ by definition. The first term on the right hand side is $O(1)$, and the second and the fourth are each bounded by $2\sup_y|\operatorname{Si}(y)|$ which is also $O(1)$. Indeed, it is a well-known fact that $\sup_{y\in\mathbf{R}}\left|\operatorname{Si}(y)\right|=O(1)$ which can be proved by observing that $t^{-1}\sin t$ is bounded on $\mathbf{R}$ and by recalling the classical result that $\lim_{y\to\infty}\operatorname{Si}(y)=\frac{\pi}{2}$. Using Lemma \ref{sincbound}, the third term can be bounded by
\begin{align}
    \label{eq:lem2bound}
    \sum_{j=j_1}^{j_{\tilde{k}}}\left|\int_{n_{j}x}^{n_{j+1}x}\frac{\sin t}{t}\,dt\right|&\leqslant \sum_{j=j_1}^{j_{\tilde{k}}}\frac{1}{n_{j}x}\min(1,|n_{j+1}x-n_jx|).
\end{align}
Recall from the definition of the indices $j_m$ that $n_{j}x\in (m,m+1]$ for all $j\in[j_m,j_{m+1})$ and hence we can bound
\begin{align*}
    \sum_{j=j_m}^{j_{m+1}-1}\frac{1}{n_{j}x}\min(1,|n_{j+1}x-n_jx|)&\leqslant \frac{1}{m}\sum_{j=j_m}^{j_{m+1}-1}\min(1,|n_{j+1}x-n_jx|)\\
    &\leqslant\frac{1}{m}\left(1+\sum_{j=j_m}^{j_{m+1}-2}|n_{j+1}x-n_jx|\right)\\
    &\leqslant \frac{1}{m}\left(1+\left((m+1)-m\right)\right)\ll \frac{1}{m}.
\end{align*}
Finally, the intervals $[j_1,j_2),[j_2,j_3),\dots,[j_{\tilde{k}-1},j_{\tilde{k}}]$ partition $[j_1,j_{\tilde{k}}]$ and hence we can bound the total sum in \eqref{eq:lem2bound} by
\begin{align*}
    \sum_{m=1}^{\tilde{k}}\sum_{j=j_m}^{j_{m+1}-1}\frac{1}{n_{j}x}\min(1,|n_{j+1}x-n_jx|)
    &\ll \sum_{m=1}^{\tilde{k}}\frac{1}{m}\\
    &\ll\log \tilde{k}\\
    &\leqslant \log \tilde{K}.
\end{align*}
We have now proved that \eqref{eq:finalstep} is bounded by $O(1+\log \tilde{K})$. 
\end{proof}
We are now in a position to prove the main result of this section.
\begin{proof}[Proof of Proposition \ref{improvetheo}]
Plugging the estimate from Lemma \ref{lem:Sibound} in \eqref{eq:L1upperbound}, we obtain $$\lVert\tilde{g}\rVert_1\ll dM(S)\log\tilde{K}.$$ In Lemma \ref{lem:L1lowerbound} we proved the lower bound $$\lVert\tilde{g}\rVert_1\gg P^{-2}\log \frac{|g(0)|}{M(S)}.$$ Combining these two bounds yields the desired inequality in Proposition \ref{improvetheo} upon recalling that $\tilde{K}=O(PK)$ by \eqref{tildeKbound}.
\end{proof}

\section{Structural results for trigonometric polynomials with few roots}
\label{sec:4}
We begin by recalling some of the set-up from Erd\'elyi and Sahasrabudhe's papers. The first step in their approach to Littlewood's problem consists of proving a structural result which in rough terms states that if $g(t) = \sum_{m=0}^N a_m \cos mt$ has $d$ roots, then the interval $[N]$ can be partitioned into $K(d)=O_d(1)$ intervals such that the coefficient sequence $(a_m)$ is periodic with period $P(d)=O_d(1)$ on each interval. The goal of this section is to prove Proposition \ref{smallperiodprop} in which we obtain an exponential improvement for the dependence of the period $P(d)$ on $d$ compared to Lemma 3.9 in \cite{erdelyi}, which itself improved on Sahasrabudhe's result \cite[Lemma 14]{sahasrabudhe} by proving that $K(d)=\exp(\exp(d^{1+o(1)}))$ instead of $\exp(\exp(d^{2+o(1)}))$. It should be noted that our work in this section is based on the ideas in \cite{erdelyi,sahasrabudhe}.

\bigskip

Let $S\subset \mathbf{Z}$ be finite and let $g$ be a cosine polynomial 
\begin{equation}
    g(t) = \sum_{m=0}^N a_m \cos mt
\label{cosinepolygene1}
\end{equation} with coefficients $a_m\in S$. Let $g$ have degree $N$ as a cosine polynomial so that we can write $2g(t)=G(e^{it})e^{-iNt}$ for a degree $2N$ polynomial $G(z)$ whose coefficients lie in $S\cup2S$.\footnote{All the coefficients of $G$ except for the coefficient of $z^N$ lie in $S$.} Throughout this section, we assume that $g$ has $d$ roots in $(0,\pi)$. We begin by collecting a number of useful lemmas.

\bigskip

Let $\boldsymbol{x}\in\mathbf{C}^B$ be a vector of length $B$ and let $1\leqslant b \leqslant B$. Borrowing terminology from \cite{sahasrabudhe}, we call the vectors
\begin{equation*}
    (x(r+1), x(r+2), \dots, x(r+b)) \in \mathbf{C}^{b}, \quad r=0,1, \dots, B-b,
\end{equation*}
the $b$-windows of $\boldsymbol{x}=(x(k))_{k=1}^B$. We need the following purely combinatorial lemma which directly follows from combining Lemmas 10 and 11 in \cite{sahasrabudhe}.

\begin{lemma}[\cite{sahasrabudhe}, Lemmas 10 and 11] Let $S\subset \mathbf{Z}$ be a finite set, let $b\in \mathbf{N}$ and let $\boldsymbol{x}\in\mathbf{R}^K$ be a sequence of length $K$ with terms in $S$. Let $1\leqslant u<v\leqslant K$ be integers and let $V$ denote the vector subspace of $\mathbf{R}^b$ spanned by the $b$-windows 
$$
(x(r+1), x(r+2), \dots, x(r+b)) \in S^{b}, \quad r=u, u+1, \dots, v-b.
$$
Suppose that $v-u>$ $|S|^{b}+3 b$ and that the non-zero vector $(c_1,c_2,\dots,c_b)$ lies in $V^{\perp}$. Then there exists a $t<b$ and $t$ sequences of length $v-u-2b+1$ that we denote by
$$(x_j(r))_{r \in[u+b, v-b]}\in \mathbf{C}^{v-u-2b+1}, \quad j=1,2, \dots, t$$
such that
$$
x(r)=x_{1}(r)+x_{2}(r)+\dots+x_{t}(r), \quad r \in[u+b, v-b]
$$
and such that the sequence $(x_j(r))_{r \in[u+b, v-b]}$ is periodic with period $p_{j} \leqslant p(b)=O(b \log \log b)$ for each $j=1,2, \dots, t$. Moreover, there exist constants $\alpha_j\in\mathbf{C}$ and roots of unity $\omega_j$ which are also roots of the polynomial $c_1+c_2z+\dots+c_bz^{b-1}$ such that for all $j\in[t]$ we have $x_j(r)=\alpha_j\omega_j^r$ when $r \in[u+b, v-b]$.
\label{rootsofunitylemma}
\end{lemma}
We also state the following result which follows from combining the proofs of Lemmas 3.7 and 3.8 in \cite{erdelyi}. Indeed, the reader should note that the only difference in the statement of our Lemma \ref{fewnonzerocoef} compared to Lemmas 3.7 and 3.8 in \cite{erdelyi} is that $Q(z)$ is allowed to be any monic polynomial of degree $\deg Q= 2d$ with coefficients of size $e^{O(d)}$ which satisfies \eqref{intetoL1}, rather than a specific choice of $Q$ with these properties. The proof of Lemma 3.7 can be used without any modification for such $Q$ while the proof of Lemma 3.8 holds because of the assumption \eqref{intetoL1} below.
\begin{lemma}[\cite{erdelyi}, Lemmas 3.7 and 3.8]
Let $g(t)=\frac{1}{2}G(e^{it})e^{-iNt}$ be a cosine polynomial with coefficients in a finite set $S\subset\mathbf{Z}$ and $d$ zeros in $(0,\pi)$. Then there exists an integer $d'=e^{O(d\log\log d)}$ such that the following holds. Let $Q(z)$ be any monic polynomial of degree $\deg Q= 2d$ with coefficients of size $e^{O(d)}$ and define the polynomial $F(z)\vcentcolon=G(z)Q(z)(z^{d'}-1)^2$. If $Q$ satisfies
\begin{align}
    \left|\int_0^{2\pi}g(t)(e^{itd'}-1)^2e^{-itd'}Q(e^{it})e^{-itd}\,dt\right|\geqslant \frac{1}{2}\int_0^{2\pi}\left|g(t)(e^{itd'}-1)^2Q(e^{it})\right|\,dt,
\label{intetoL1}
\end{align}
then $F(z)$ has at most $q\leqslant \exp((2M(S))^{O(d\log d)})$ non-zero coefficients.
\label{fewnonzerocoef}
\end{lemma}
In Erd\'elyi and Sahasrabudhe's papers $Q$ is chosen to be the ‘companion' polynomial of $g$ which, if the sign changes of $g$ in $(0,\pi)$ occur at $t_1<t_2<\dots<t_d$, is defined by $Q_0(e^{it})e^{-itd}\vcentcolon=2^d\prod_{j=1}^d(\cos t-\cos t_j)$ and is the unique monic degree $d$ cosine polynomial with sign changes at all $t_j$. It is clear that for this choice of $Q_0$ the integrand in \eqref{intetoL1} is a real-valued cosine polynomial with no sign changes so that \eqref{intetoL1} holds even when the constant $1/2$ on the right hand side is replaced by $1$. For technical reasons, we use the perturbation
\begin{align}
    Q_\varepsilon(e^{it})e^{-itd}\vcentcolon=2^d\prod_{j=1}^d(\cos t-\cos (t_j+\varepsilon)),
\end{align} where $\varepsilon$ is a constant to be chosen so that \eqref{intetoL1} holds and such that none of the complex roots of $Q_\varepsilon(z)$ are roots of unity. The first property holds whenever $\varepsilon$ is sufficiently small because $Q_0$ being the companion polynomial of $g$ implies that
\begin{equation*}
    \left|\int_0^{2\pi}g(t)(e^{itd'}-1)^2e^{-itd'}Q_\varepsilon(e^{it})e^{-itd}\,dt\right|\geqslant \left\lVert g(t)(e^{itd'}-1)^2e^{-itd'}Q_\varepsilon(e^{it})e^{-itd}\right\rVert_1
\end{equation*} for $\varepsilon=0$, and because both sides of this inequality depend continuously on $\varepsilon$. Clearly we can choose arbitrarily small such $\varepsilon$ such that $\varepsilon+t_j\notin \pi\cdot\mathbf{Q}$ for all $j$ and hence no roots of $Q_\varepsilon$ are roots of unity. As $Q_\varepsilon$ has degree $2d$ and coefficients of size $e^{O(d)}$, an application of Lemma \ref{fewnonzerocoef} yields the following corollary.
\begin{cor}
    Let $g(t)=\frac{1}{2}G(e^{it})e^{-iNt}$ be a cosine polynomial with coefficients in a finite set $S\subset\mathbf{Z}$ and $d$ zeros in $(0,\pi)$. Then there exists an integer $d'=e^{O(d\log\log d)}$ and a monic polynomial $Q(z)$ of degree $2d$ none of whose complex roots are roots of unity such that the following holds. The polynomial $F(z)\vcentcolon=G(z)Q(z)(z^{d'}-1)^2$
has at most $q\leqslant \exp((2M(S))^{O(d\log d)})$ non-zero coefficients.
\label{fewnonzerocoefcor}
\end{cor}
In the final step of this section, we make a modest improvement over Erd\'elyi and Sahasrabudhe's structural result for trigonometric polynomials with ‘few' zeros by showing that Corollary \ref{fewnonzerocoefcor} implies that we can partition $[N]=\cup_{j=1}^K I_j$ into $K=\exp((2M(S))^{O(d\log d)})$ intervals such that on each interval $I_j$, the coefficient sequence $(a_n)_{n\in I_j}$ is periodic with period $P=e^{O(d\log\log d)}$. This removes one exponential from Erd\'elyi's bound for the period $P$ which in his paper (see Lemma 3.9 in \cite{erdelyi}) is essentially a double exponential in $d$.
\begin{proposition}
Let $S\subset \mathbf{Z}$ be a finite set and $M(S)=\max_{s\in S} |s|$. Let $g(t) = \sum_{n=0}^N a_n \cos nt$ be a cosine polynomial with coefficients $a_n\in S$. Suppose that $\deg g= N$ and that $g$ has $d$ zeros in $(0,\pi)$. Then there exists an integer $P=e^{O(d\log\log d)}$ and a partition $[0,N]=\cup_{j=1}^KI_j$ of $[0,N]$ into $K=\exp((2M(S))^{O(d\log d)})$ intervals such that on each interval $I_j$, the coefficient sequence $(a_n)_{n\in I_j}$ is periodic with period $P$. 
\label{smallperiodprop}
\end{proposition}
Before providing the details, we give a quick overview of the ideas that go into the proof. Suppose that we have a trigonometric polynomial $G$ with coefficients in $S$ and at most $d$ zeros. Corollary \ref{fewnonzerocoefcor} shows that one can find a polynomial $\tilde{G}$ whose degree and coefficients are bounded by constants that only depend on $d$ and $S$, and which is strongly correlated with $G$ in the following sense. The product $G\tilde{G}$ is a polynomial with a bounded (also only depending on $d,S$) number of terms. Whenever a coefficient of $G\tilde{G}$ is zero, one can naturally interpret this as saying that the coefficient sequence of $G$ satisfies a recurrence relation provided by the coefficients of $\tilde{G}$. Since all but a bounded number of coefficients of $G\tilde{G}$ are non-zero, one can split the coefficient sequence $(a_n)$ of $G$ into a bounded number of pieces, such that on each piece it satisfies a recurrence of length $\deg \tilde{G}=O_{d,S}(1)$. Since the $a_n$ lie in a finite set $S$, a combinatorial result such as Lemma \ref{rootsofunitylemma} can be used to deduce that $(a_n)$ is in fact periodic of period $O_{d,S}(1)$ on each such piece.
\begin{proof}
    We may write $2g(t)=G(e^{it})e^{-iNt}$ for a polynomial $G(z)=\sum_{m=0}^{2N}b_mz^m$ with
    \begin{equation}
        b_m=a_{m-N}, \quad m\in[N+1,2N].
    \label{Gdefi}
    \end{equation} By Corollary \ref{fewnonzerocoefcor}, there exists another polynomial 
    \begin{equation}
    \tilde{G}(z)=Q(z)(z^{d'}-1)^2
    \label{Gtildedefi}
    \end{equation}
    such that $F(z)\vcentcolon=G(z)\tilde{G}(z)$ has 
    \begin{equation}
    q\leqslant\exp((2M(S))^{O(d\log d)})
    \label{qboundinprop}
    \end{equation} non-zero coefficients. We also note from Corollary \ref{fewnonzerocoefcor} that we may assume that the polynomial $Q$ has degree $\deg Q=O(d)$, that none of the complex roots of $Q$ are roots of unity and that $d'$ is a positive integer of size \begin{equation}
    d'=\exp(O(d\log \log d)).
    \label{d'bound}
    \end{equation}
    Let us write $\tilde{G}(z)=\sum_{m=0}^{D}c_{D-m}z^m$
where $D=2d'+\deg Q=\exp(O(d\log\log d))$ and let $\boldsymbol{b}=(b_m)_{m=0}^{2N}$ denote the coefficient sequence of $G$. Then $F=G\tilde{G}$ having $q$ non-zero coefficients implies that we may find intervals $J_1,J_2,\dots,J_{K_0}$ which form a partition $[N,2N]=\cup_{j=1}^{K_0}J_i$ into $K_0\leqslant q$ parts such that whenever $(b_u,b_{u+1},\dots,b_{u+D})$ is a $(D+1)$-window of $\boldsymbol{b}$ for which $[u,u+D]$ is fully contained in one of the intervals $J_i$, then 
    \begin{align}
        (b_u,b_{u+1},\dots,b_{u+D})\cdot (c_0,c_1,\dots,c_D)=0.
    \label{bperp}
    \end{align}
    To see this, simply note that if the dot product in \eqref{bperp} is non-zero then $F(z)$ has a non-zero coefficient at $z^{u+D}$. Consider first all the intervals $J_i$ whose length is at most $(2M(S))^{O(D)}$, which we will refer to as ‘short' intervals. For each such ‘short' interval $J_i$ we simply partition it into $|J_i|$ many intervals $I^{(i)}_1,\dots,I^{(i)}_{|J_i|}$ of length 1. The intervals $I^{(i)}_j-N$ will be parts\footnote{We need to shift these intervals by $-N$ by \eqref{Gdefi}.} of the desired partition $I_1,I_2,\dots,I_K$ of $[0,N]$ from the conclusion of Proposition \ref{smallperiodprop} and note that as they all have length 1, the required property that the coefficient sequence $(a_n)$ of $g$ is $P$-periodic on $I^{(i)}_j-N$ is vacuously satisfied. Moreover, note that the total number of intervals $I^{(i)}_j$ that are obtained from ‘short' $J_i$ is 
    \begin{align}
    O(q\max_{J_i\text{ is ‘short'}}|J_i|)=O(q(2M(S))^{O(D)})=O(\exp((2M(S))^{O(d\log\log d)}))
    \label{fewparts1}
    \end{align}
    recalling that $D=\exp(O(d\log\log d))$ by \eqref{d'bound} and $q\leqslant\exp((2M(S))^{O(d\log d)})$ by \eqref{qboundinprop}. We can now consider the ‘long' intervals $J_i$ which have length at least $(2M(S))^{O(D)}$ and, if we choose the absolute constant hidden in the $O$-notation to be sufficiently large, then the sequence $(b_m)_{m\in J_i}$ satisfies the assumptions of Lemma \ref{rootsofunitylemma} as \eqref{bperp} is also satisfied on $(b_m)_{m\in J_i}$. Let us subdivide each ‘long' interval $J_i=[x_i,y_i]$ into $I^{(i)}=[x_i+2D,y_i-2D]$ and $O(D)$ many intervals of length 1 that we denote by $I^{(i)}_j$. Then the intervals $I^{(i)}_j-N$ and $I^{(i)}-N$ coming from ‘long' intervals $J_i$ complete the desired partition of $[0,N]$ from the conclusion of Proposition \ref{smallperiodprop}. To finish, we prove the following claims:
    \begin{itemize}
        \item [(i)] The intervals $\{I^{(i)}_j-N:J_i\text{ is ‘short'}\}\cup\{I^{(i)}_j-N:J_i\text{ is ‘long'}\}\cup\{I^{(i)}-N:J_i\text{ is ‘long'}\}$ partition $[0,N]$.
        \item [(ii)] This partition has at most $\exp((2M(S))^{O(d\log\log d)})$ parts.
        \item [(iii)] There is a positive integer $P=\exp(O(d\log\log d))$ such that if $I$ is an interval in the partition, then the coefficient sequence $(a_m)_{m\in I}$ of $g$ is $P$-periodic on $I$.
    \end{itemize}
    Claim (i) is clear since the intervals $J_i$ partition $[N,2N]$ and the intervals appearing in (i) are simply subdivisions of these $J_i$. For (ii), we recall from \eqref{fewparts1} that the number of $I^{(i)}_j$ coming from ‘short' $J_i$ is $O(\exp((2M(S))^{O(d\log\log d)}))$ as desired. Similarly, each ‘long' $J_i$ was subdivided into at most $O(D)$ intervals of the form $I^{(i)}_j$ or $I^{(i)}$ so these also contribute at most $O(Dq)=\exp((2M(S))^{O(d\log\log d)})$ parts to the partition, recalling \eqref{qboundinprop} and that $D=\exp(O(d\log\log d))$. Finally, claim (iii) is vacuously true for each of the intervals $I^{(i)}_j-N$ for any ‘long' or ‘short' $J_i$ since all these intervals have length 1. To finish, it remains to show that the sequences $(a_m)_{m\in I^{(i)}-N}$ are $P$-periodic when $i$ is an index for which $J_i$ is a ‘long' interval. By our discussion above, whenever $J_i=[x_i,y_i]$ is a ‘long' interval we may apply Lemma \ref{rootsofunitylemma} to the sequence $(b_m)_{m\in J_i}$ and hence there is a $t\leqslant D$, constants $\alpha_j$ and roots of unity $\omega_j$ which are also roots of the polynomial $\sum_{m=0}^Dc_mz^m=z^D\tilde{G}(1/z)$ such that for all $m\in[x_i+2D,y_i-2D]=\vcentcolon I^{(i)}$ we have
    \begin{equation}
        b_m=\sum_{j=1}^t\alpha_j\omega_j^m.
    \label{periodicb}
    \end{equation}
    We will show that the positive integer $P\vcentcolon=d'$ has the property that any root of unity $\omega$ which is also a root of $z^D\tilde{G}(1/z)$ satisfies $\omega^P=1$, and hence $(b_m)$ is $P$-periodic on each interval $I^{(i)}$ so that $(a_m)_{m\in I^{(i)}-N}$ is also $P$-periodic by \eqref{Gdefi}. As $d'=\exp(O(d\log\log d))$ by \eqref{d'bound}, this would give the required bound for $P$. Note that by its definition \eqref{Gtildedefi} we have $z^D\tilde{G}(1/z)=(z^{d'}-1)^2Q^*(z)$ where $Q^*(z)=z^{\deg Q}Q(1/z)$ is the reciprocal polynomial of $Q$. By Corollary \ref{fewnonzerocoefcor}, none of the roots of $Q$ are roots of unity and hence neither are the roots of $Q^*$. Hence, if $\omega$ is a root of unity for which $\tilde{G}(1/\omega)=0$, then $\omega^{d'}=1$ as desired.
\end{proof}

\section{Proof of the main theorem}
\label{sec:5}
In this short section, we show how the results from Sections 3 and 4 may be combined to prove Theorem \ref{newtheorem}. 
\begin{proof}[Proof of Theorem \ref{newtheorem}]Let $S\subset \mathbf{Z}$ be a finite set. Let $g(t)=\sum_{n=0}^N a_n\cos nt$ be a cosine polynomial with coefficients $a_n\in S$ and $d$ roots in $(0,\pi)$. Then Proposition \ref{smallperiodprop} implies that $[0,N]$ can be partitioned into $K=\exp((2M(S))^{O(d\log d)})$ intervals such that $(a_m)$ is periodic with period $P$ on each interval, and $P=\exp(O(d\log\log d))$. Using these bounds for $K$ and $P$ in Proposition \ref{improvetheo} gives
\begin{align*}
     \log|g(0)| &\ll dM(S)\exp(O(d\log\log d))(2M(S))^{O(d\log d)}\\
     &\ll \exp(O(d\log d (1+\log M(S))))
\end{align*}
which simplifies to
\begin{align*}
    d\log d\gg \frac{1}{1+\log M(S)}\log\log |g(0)|.
\end{align*}
This clearly implies the bound that we claimed in Theorem \ref{newtheorem}.
\end{proof}
\bibliographystyle{plain}

\bigskip

\noindent
{\sc Mathematical Institute, Andrew Wiles Building, University of Oxford, Radcliffe
Observatory Quarter, Woodstock Road, Oxford, OX2 6GG, UK.}\newline
\href{mailto:benjamin.bedert@maths.ox.ac.uk}{\small benjamin.bedert@maths.ox.ac.uk}
\end{document}